\documentclass{amsart}
\usepackage{amssymb}
\usepackage{amsmath}
\usepackage{amsfonts}

\newtheorem{theorem}{Theorem}
\theoremstyle{plain}

\newtheorem{definition}{Definition}

\newtheorem{lemma}{Lemma}

\numberwithin{equation}{section}
\input{tcilatex}

\begin{document}
\title[Two New Convex Dominated Functions]{Two New Different Kinds of Convex
Dominated Functions and Inequalities via Hermite-Hadamard Type}
\author{M.Emin \"{O}zdemir$^{\blacktriangle }$}
\address{$^{\blacktriangle }$Atat\"{u}rk University, K.K. Education Faculty,
Department of Mathematics, 25240, Campus, Erzurum, Turkey}
\email{emos@atauni.edu.tr}
\author{Mevl\"{u}t Tun\c{c}}
\address{Department of Mathematics, Faculty of Art and Sciences, Kilis 7
Aralik University, Kilis, 79000, Turkey}
\email{mevluttunc@kilis.edu.tr}
\author{Havva Kavurmac\i $^{\blacktriangle ,\blacksquare }$}
\email{hkavurmaci@atauni.edu.tr}
\thanks{$^{\blacksquare }$Corresponding Author}
\date{February 9, 2012}

\begin{abstract}
In this paper, we establish two new convex dominated function and then we
obtain new Hadamard type inequalities related to this definitions.
\end{abstract}

\maketitle

\section{INTRODUCTION}

Let $f:I\subseteq 
\mathbb{R}
\rightarrow 
\mathbb{R}
$ be a convex function and let $a,b\in I,$ with $a<b.$ The following
inequality%
\begin{equation*}
\ \ \ f\left( \frac{a+b}{2}\right) \leq \frac{1}{b-a}\int_{a}^{b}f(x)dx\leq 
\frac{f(a)+f(b)}{2}
\end{equation*}%
is known in the literature as Hadamard's inequality. Both inequalities hold
in the reversed direction if $f$ is concave.

In \cite{GL}, Godunova and Levin introduced the following class of functions.

\begin{definition}
A function $f:I\subseteq 
\mathbb{R}
\rightarrow 
\mathbb{R}
$ is said to belong to the class of $Q(I)$ if it is nonnegative and for all $%
x,y\in I$ and $\lambda \in (0,1)$ satisfies the inequality;%
\begin{equation}
f(\lambda x+(1-\lambda )y)\leq \frac{f(x)}{\lambda }+\frac{f(y)}{1-\lambda }.
\label{1}
\end{equation}
\end{definition}

They also noted that all nonnegative monotonic and nonnegative convex
functions belong to this class and also proved the following motivating
result:

If $f\in Q(I)$ and $x,y,z\in I,$ then%
\begin{equation}
f(x)(x-y)(x-z)+f(y)(y-x)(y-z)+f(z)(z-x)(z-y)\geq 0.  \label{2}
\end{equation}%
In fact (\ref{1}) is even equivalent to (\ref{2}). So it can alternatively
be used in the definition of the class $Q(I).$

In \cite{DPP1}, Dragomir et al. defined the following new class of functions.

\begin{definition}
A function $f:I\subseteq 
\mathbb{R}
\rightarrow 
\mathbb{R}
$ is $P$ function or that $f$ belongs to the class of $P(I),$ if it is
nonnegative and for all $x,y\in I$ and $\lambda \in \lbrack 0,1],$ satisfies
the following inequality;%
\begin{equation*}
f(\lambda x+(1-\lambda )y)\leq f(x)+f(y).
\end{equation*}
\end{definition}

In \cite{DPP1}, Dragomir et al. proved the following inequalities of
Hadamard type for class of $Q(I)-$ functions and $P-$ functions.

\begin{theorem}
Let $f\in Q(I),$ $a,b\in I$ with $a<b$ and $f\in L_{1}[a,b].$ Then the
following inequalities hold:%
\begin{equation*}
\ \ \ f\left( \frac{a+b}{2}\right) \leq \frac{4}{b-a}\int_{a}^{b}f(x)dx
\end{equation*}%
and%
\begin{equation*}
\frac{1}{b-a}\int_{a}^{b}p\left( x\right) f(x)dx\leq \frac{f\left( a\right)
+f\left( b\right) }{2}
\end{equation*}%
where $p\left( x\right) =\frac{\left( b-x\right) \left( x-a\right) }{\left(
b-a\right) ^{2}},$ $x\in \left[ a,b\right] .$
\end{theorem}

\begin{theorem}
Let $f\in P(I),$ $a,b\in I$ with $a<b$ and $f\in L_{1}[a,b].$ Then the
following inequality holds:%
\begin{equation*}
f\left( \frac{a+b}{2}\right) \leq \frac{2}{b-a}\int_{a}^{b}f(x)dx\leq
2[f(a)+f(b)].
\end{equation*}
\end{theorem}

In \cite{DI} and \cite{DPP2}, the authors connect together some disparate
threads through a Hermite-Hadamard motif. The first of these threads is the
unifying concept of a $g-$convex-dominated function. In \cite{HHW}, Hwang et
al. established some inequalities of Fej\'{e}r type for $g-$convex-dominated
functions. Finally, in \cite{KOS} Kavurmac\i\ et al. introduced several new
different kinds of convex -dominated functions and then gave H-H type
inequalities for this classes of functions.

The main purpose of this paper is to introduce two new convex-dominated
function and then present new H-H type inequalities related to these
definitions.

\section{$\left( g\text{, }Q(I)\right) -$convex dominated functions}

\begin{definition}
Let a nonnegative function $g:I\subseteq 
\mathbb{R}
\rightarrow 
\mathbb{R}
$ belong to the class of $Q(I).$ The real function $f:I\subseteq 
\mathbb{R}
\rightarrow 
\mathbb{R}
$ is called $\left( g\text{, }Q(I)\right) -$convex dominated on $I$ if the
following condition is satisfied;%
\begin{eqnarray}
&&\left\vert \frac{f\left( x\right) }{\lambda }+\frac{f\left( y\right) }{%
1-\lambda }-f\left( \lambda x+\left( 1-\lambda \right) y\right) \right\vert 
\label{h1} \\
&&  \notag \\
&\leq &\frac{g\left( x\right) }{\lambda }+\frac{g\left( y\right) }{1-\lambda 
}-g\left( \lambda x+\left( 1-\lambda \right) y\right)   \notag
\end{eqnarray}%
for all $x,y\in I$ and $\lambda \in \left( 0,1\right) $.
\end{definition}

The next simple characterisation of $\left( g\text{, }Q(I)\right) -$convex
dominated functions holds.

\begin{lemma}
Let a nonnegative function $g:I\subseteq 
\mathbb{R}
\rightarrow 
\mathbb{R}
$ belong to the class of $Q(I)$ and $f:I\subseteq 
\mathbb{R}
\rightarrow 
\mathbb{R}
$ be a real function. The following statements are equivalent:
\end{lemma}

\begin{enumerate}
\item $f$ is $\left( g\text{, }Q(I)\right) -$convex dominated on $I.$

\item The mappings $g-f$ and $g+f$ are $\left( g\text{, }Q(I)\right) -$
convex on $I.$

\item There exist two $\left( g\text{, }Q(I)\right) -$convex mappings $l,k$
defined on $I$ such that%
\begin{equation*}
\begin{array}{ccc}
f=\frac{1}{2}\left( l-k\right) & \text{and} & g=\frac{1}{2}\left( l+k\right)%
\end{array}%
.
\end{equation*}
\end{enumerate}

\begin{proof}
1$\Longleftrightarrow $2 The condition (\ref{h1}) is equivalent to%
\begin{eqnarray*}
&&g\left( \lambda x+\left( 1-\lambda \right) y\right) -\frac{g\left(
x\right) }{\lambda }-\frac{g\left( y\right) }{1-\lambda } \\
&\leq &\frac{f\left( x\right) }{\lambda }+\frac{f\left( y\right) }{1-\lambda 
}-f\left( \lambda x+\left( 1-\lambda \right) y\right) \\
&\leq &\frac{g\left( x\right) }{\lambda }+\frac{g\left( y\right) }{1-\lambda 
}-g\left( \lambda x+\left( 1-\lambda \right) y\right)
\end{eqnarray*}%
for all $x,y\in I$ and $\lambda \in \left( 0,1\right) .$ The two
inequalities may be rearranged as%
\begin{equation*}
\left( g+f\right) \left( \lambda x+\left( 1-\lambda \right) y\right) \leq 
\frac{\left( g+f\right) (x)}{\lambda }+\frac{\left( g+f\right) (y)}{%
1-\lambda }
\end{equation*}%
and%
\begin{equation*}
\left( g-f\right) \left( \lambda x+\left( 1-\lambda \right) y\right) \leq 
\frac{\left( g-f\right) (x)}{\lambda }+\frac{\left( g-f\right) (y)}{%
1-\lambda }
\end{equation*}%
which are equivalent to the $\left( g\text{, }Q(I)\right) -$convexity of $%
g+f $ and $g-f,$ respectively.

2$\Longleftrightarrow $3 Let we define the mappings $f,$ $g$ as $f=\frac{1}{2%
}\left( l-k\right) $ and $g=\frac{1}{2}\left( l+k\right) $. Then if we sum
and subtract $f,$ $g,$ respectively, we have $g+f=l$ and $g-f=k.$ By the
condition 2 of Lemma 1, the mappings $g-f$ and $g+f$ are $\left( g\text{, }%
Q(I)\right) -$convex on $I,$ so $l,k$ are $\left( g\text{, }Q(I)\right) -$%
convex mappings on $I$ too.
\end{proof}

\begin{theorem}
Let a nonnegative function $g:I\subseteq 
\mathbb{R}
\rightarrow 
\mathbb{R}
$ belong to the class of $Q(I)$ and the real function $f:I\subseteq 
\mathbb{R}
\rightarrow 
\mathbb{R}
$ is $\left( g\text{, }Q(I)\right) -$convex dominated on $I.$ If $a,b\in I$
with $a<b$ and $f,g\in L_{1}[a,b],$ then one has the inequalities:%
\begin{equation*}
\left\vert \frac{4}{b-a}\int_{a}^{b}f\left( x\right) dx-f\left( \frac{a+b}{2}%
\right) \right\vert \leq \frac{4}{b-a}\int_{a}^{b}g\left( x\right)
dx-g\left( \frac{a+b}{2}\right) 
\end{equation*}%
and%
\begin{equation*}
\left\vert \frac{f\left( a\right) +f\left( b\right) }{2}-\frac{1}{b-a}%
\int_{a}^{b}p\left( x\right) f\left( x\right) dx\right\vert \leq \frac{%
g\left( a\right) +g\left( b\right) }{2}-\frac{1}{b-a}\int_{a}^{b}p\left(
x\right) g\left( x\right) dx
\end{equation*}%
for all $x,y\in I$ and $p\left( x\right) $ as in Theorem 1$.$
\end{theorem}

\begin{proof}
By Definition 1 with $\lambda =\frac{1}{2},\ x=ta+(1-t)b,\ y=\left(
1-t\right) a+tb\ $and$\ t\in \left[ 0,1\right] ,$ as the mapping $f$ is $%
\left( g\text{, }Q(I)\right) -$convex dominated function, we have that%
\begin{eqnarray*}
&&\left\vert 2\left[ f\left( ta+\left( 1-t\right) b\right) +f\left( \left(
1-t\right) a+tb\right) \right] -f\left( \frac{a+b}{2}\right) \right\vert \\
&\leq &2\left[ g\left( ta+\left( 1-t\right) b\right) +g\left( \left(
1-t\right) a+tb\right) \right] -g\left( \frac{a+b}{2}\right) .
\end{eqnarray*}%
Integrating the above inequality over $t$ on $\left[ 0,1\right] ,$ the first
inequality is proved.

Since $f$ is $\left( g\text{, }Q(I)\right) -$convex dominated using
Definition 1 with$\ x=a,\ y=b\ $and$\ t\in \left[ 0,1\right] ,$ we can write%
\begin{eqnarray*}
&&\left\vert \left( 1-t\right) f\left( a\right) +tf\left( b\right) -t\left(
1-t\right) f\left( ta+\left( 1-t\right) b\right) \right\vert \\
&\leq &\left( 1-t\right) g\left( a\right) +tg\left( b\right) -t\left(
1-t\right) g\left( ta+\left( 1-t\right) b\right)
\end{eqnarray*}%
and%
\begin{eqnarray*}
&&\left\vert tf\left( a\right) +\left( 1-t\right) f\left( b\right) -t\left(
1-t\right) f\left( \left( 1-t\right) a+tb\right) \right\vert \\
&\leq &tg\left( a\right) +\left( 1-t\right) g\left( b\right) -t\left(
1-t\right) g\left( \left( 1-t\right) a+tb\right) .
\end{eqnarray*}%
Then, adding above inequalities we have 
\begin{eqnarray*}
&&\left\vert \left[ f\left( a\right) +f\left( b\right) \right] -t\left(
1-t\right) \left[ f\left( ta+\left( 1-t\right) b\right) +f\left( \left(
1-t\right) a+tb\right) \right] \right\vert \\
&\leq &\left[ g\left( a\right) +g\left( b\right) \right] -t\left( 1-t\right) %
\left[ g\left( ta+\left( 1-t\right) b\right) +g\left( \left( 1-t\right)
a+tb\right) \right] .
\end{eqnarray*}%
Integrating the resulting inequality over $t$ on $\left[ 0,1\right] ,$ we
get the second inequality. The proof is completed.
\end{proof}

\section{$\left( g\text{, }P(I)\right) -$convex dominated functions}

\begin{definition}
Let a nonnegative function $g:I\subseteq 
\mathbb{R}
\rightarrow 
\mathbb{R}
$ belong to the class of $P(I).$ The real function $f:I\subseteq 
\mathbb{R}
\rightarrow 
\mathbb{R}
$ is called $\left( g\text{, }P(I)\right) -$convex dominated on $I$ if the
following condition is satisfied;%
\begin{eqnarray}
&&\left\vert \left[ f(x)+f(y)\right] -f\left( \lambda x+\left( 1-\lambda
\right) y\right) \right\vert   \label{h2} \\
&&  \notag \\
&\leq &\left[ g(x)+g(y)\right] -g\left( \lambda x+\left( 1-\lambda \right)
y\right)   \notag
\end{eqnarray}%
for all $x,y\in I$ and $\lambda \in \left[ 0,1\right] $.
\end{definition}

The next simple characterisation of $\left( g\text{, }P(I)\right) -$convex
dominated functions holds.

\begin{lemma}
Let a nonnegative function $g:I\subseteq 
\mathbb{R}
\rightarrow 
\mathbb{R}
$ belong to the class of $P(I)$ and $f:I\subseteq 
\mathbb{R}
\rightarrow 
\mathbb{R}
$ be a real function. The following statements are equivalent:
\end{lemma}

\begin{enumerate}
\item $f$ is $\left( g\text{, }P(I)\right) -$convex dominated on $I.$

\item The mappings $g-f$ and $g+f$ are $\left( g\text{, }P(I)\right) -$
convex on $I.$

\item There exist two $\left( g\text{, }P(I)\right) -$convex mappings $l,k$
defined on $I$ such that%
\begin{equation*}
\begin{array}{ccc}
f=\frac{1}{2}\left( l-k\right) & \text{and} & g=\frac{1}{2}\left( l+k\right)%
\end{array}%
.
\end{equation*}
\end{enumerate}

\begin{proof}
1$\Longleftrightarrow $2 The condition (\ref{h2}) is equivalent to%
\begin{eqnarray*}
&&g\left( \lambda x+\left( 1-\lambda \right) y\right) -\left[ g\left(
x\right) +g\left( y\right) \right] \\
&\leq &\left[ f\left( x\right) +f\left( y\right) \right] -f\left( \lambda
x+\left( 1-\lambda \right) y\right) \\
&\leq &\left[ g\left( x\right) +g\left( y\right) \right] -g\left( \lambda
x+\left( 1-\lambda \right) y\right)
\end{eqnarray*}%
for all $x,y\in I$ and $\lambda \in \left[ 0,1\right] .$ The two
inequalities may be rearranged as%
\begin{equation*}
\left( g+f\right) \left( \lambda x+\left( 1-\lambda \right) y\right) \leq
\left( g+f\right) (x)+\left( g+f\right) (y)
\end{equation*}%
and%
\begin{equation*}
\left( g-f\right) \left( \lambda x+\left( 1-\lambda \right) y\right) \leq
\left( g-f\right) (x)+\left( g-f\right) (y)
\end{equation*}%
which are equivalent to the $\left( g\text{, }P(I)\right) -$convexity of $%
g+f $ and $g-f,$ respectively.

2$\Longleftrightarrow $3 Let we define the mappings $f,$ $g$ as $f=\frac{1}{2%
}\left( l-k\right) $ and $g=\frac{1}{2}\left( l+k\right) $. Then if we sum
and subtract $f,$ $g,$ respectively, we have $g+f=l$ and $g-f=k.$ By the
condition 2 of Lemma 2, the mappings $g-f$ and $g+f$ are $\left( g\text{, }%
P(I)\right) -$convex on $I,$ so $l,k$ are $\left( g\text{, }P(I)\right) -$%
convex mappings on $I$ too.
\end{proof}

\begin{theorem}
Let a nonnegative function $g:I\subseteq 
\mathbb{R}
\rightarrow 
\mathbb{R}
$ belong to the class of $P(I).$ The real function $f:I\subseteq 
\mathbb{R}
\rightarrow 
\mathbb{R}
$ is $\left( g\text{, }P(I)\right) -$convex dominated on $I.$ If $a,b\in I$
with $a<b$ and $f,g\in L_{1}[a,b],$ then one has the inequalities:%
\begin{equation*}
\left\vert \frac{2}{b-a}\int_{a}^{b}f\left( x\right) dx-f\left( \frac{a+b}{2}%
\right) \right\vert \leq \frac{2}{b-a}\int_{a}^{b}g\left( x\right)
dx-g\left( \frac{a+b}{2}\right) 
\end{equation*}%
and%
\begin{equation*}
\left\vert \left[ f\left( a\right) +f\left( b\right) \right] -\frac{1}{b-a}%
\int_{a}^{b}f\left( x\right) dx\right\vert \leq \left[ g\left( a\right)
+g\left( b\right) \right] -\frac{1}{b-a}\int_{a}^{b}g\left( x\right) dx
\end{equation*}%
for all $x,y\in I.$
\end{theorem}

\begin{proof}
By Definition 4 with $\lambda =\frac{1}{2},\ x=ta+(1-t)b,\ y=\left(
1-t\right) a+tb\ $and$\ t\in \left[ 0,1\right] ,$ as the mapping $f$ is $%
\left( g\text{, }P(I)\right) -$convex dominated function, we have%
\begin{eqnarray*}
&&\left\vert \left[ f\left( ta+\left( 1-t\right) b\right) +f\left( \left(
1-t\right) a+tb\right) \right] -f\left( \frac{a+b}{2}\right) \right\vert \\
&\leq &\left[ g\left( ta+\left( 1-t\right) b\right) +g\left( \left(
1-t\right) a+tb\right) \right] -g\left( \frac{a+b}{2}\right) .
\end{eqnarray*}%
Integrating the above inequality over $t$ on $\left[ 0,1\right] ,$ the first
inequality is proved.

Since $f$ is $\left( g\text{, }P(I)\right) -$convex dominated using
Definition 4 with$\ x=a,\ y=b\ $and$\ t\in \left[ 0,1\right] ,$ we can write%
\begin{eqnarray*}
&&\left\vert \left[ f\left( a\right) +f\left( b\right) \right] -f\left(
ta+\left( 1-t\right) b\right) \right\vert \\
&\leq &g\left( a\right) +g\left( b\right) -g\left( ta+\left( 1-t\right)
b\right) .
\end{eqnarray*}%
Integrating the above resulting inequality over $t$ on $\left[ 0,1\right] $
we get the second inequality. The proof is completed.
\end{proof}

\end{document}